\documentclass[1p,nopreprintline]{elsarticle}

\usepackage{lineno,hyperref}
\modulolinenumbers[5]

\journal{ }

\usepackage{amsmath,amssymb,amsfonts,amsthm}
\usepackage{algorithmic}
\usepackage{textcomp}
\usepackage{bm,graphicx}

\newtheorem{theorem}{Theorem}

\newtheorem{example}{Example}
\newtheorem{corollary}{Corollary}

\newproof{pf}{Proof}

\begin{document}
	
	\begin{frontmatter}
		
		\title{Convex optimization of bioprocesses}
		
		\author{Josh A. Taylor}
		\ead{josh.taylor@utoronto.ca}
		\address{The Edward S. Rogers Sr. Department of Electrical\\ \& Computer Engineering, University of Toronto, Toronto, Canada}
		
		\author{Alain Rapaport}
		\ead{alain.rapaport@inrae.fr}
		\address{MISTEA, Univ. Montpellier, INRAE, Institut Agro, Montpellier, France}
		%
		
		\author{Denis Dochain}
		\ead{denis.dochain@uclouvain.be}
		\address{ICTEAM, Universit\'{e} Catholique de Louvain, Belgium}
		%

		%
		%
		
		\begin{abstract}
			We optimize a general model of bioprocesses, which is nonconvex due to the microbial growth in the biochemical reactors. We formulate a convex relaxation and give conditions guaranteeing its exactness in both the transient and steady state cases. When the growth kinetics are modeled by the Monod function under constant biomass or the Contois function, the relaxation is a second-order cone program, which can be solved efficiently at large scales. We implement the model on a numerical example based on a wastewater treatment system.
		\end{abstract}
		
		
	\end{frontmatter}

\section{Introduction}\label{intro}
We optimize a model of dynamical bioprocesses consisting of a set of biochemical reactors interconnected by diffusion and mass flow. The objectives include minimizing substrate outflow, maximizing biogas production, and tracking setpoints. Within each reactor, several microbial reactions convert any number of biotic or abiotic reactants into biomass and/or products. This setup describes a variety of physical systems such as wastewater treatment networks, the production of various chemicals, and compartmental approximations of bioprocesses in continuous media. Here we focus on the case when the kinetics can be represented by a second-order cone (SOC) constraint, as recently shown in~\cite{taylor2021grad} for the Monod~\cite{monod1949growth} and Contois~\cite{contois1959kinetics} growth rates.

The most closely related topics to ours are chemical process optimization, control and optimization of wastewater systems, and control of bioprocesses. Most existing approaches to process optimization~\cite{biegler2010nonlinear} and wastewater~\cite{ocampo2010model} do not explicitly model the microbial growth, and often use either linear programming or general nonlinear solvers. There have been many applications of nonlinear control~\cite{bastin2013line} and optimization~\cite{srinivasan2003dynamic} to bioprocesses, but not convex relaxations or second-order cone programming (SOCP).

Our main results are generalizations of those in~\cite{taylor2021grad}, which focused on the gradostat with a single reaction~\cite{smith1995theory}. Here we allow for any number of substrates and biomasses, general convex objectives, and multiple biochemical reactions. Our original theoretical contributions are as follows.
\begin{itemize}
\item In Section~\ref{relax}, we formulate a convex relaxation for optimizing the trajectory and steady state solution of a general bioprocess.
\item In Section~\ref{exact}, we give conditions under which the relaxations are guaranteed to be exact in both the transient and steady state cases.
\end{itemize}

To streamline exposition, the only external inputs to the model are the influent concentrations, e.g., biochemical oxygen demand and ammonia. Our main exactness results straightforwardly apply when the flow rates are also variable, but the resulting bilinearities make the problem nonconvex. This can be handled using techniques like disjunctive programming, as in~\cite{taylor2021grad}, or further convex relaxation, which we discuss at the end of Section~\ref{statement}.

We apply our results in two examples. In Section~\ref{app:gradostat}, we show that our exactness conditions simplify to those in~\cite{taylor2021grad} when specialized to the gradostat. In Section~\ref{app:ww}, we optimize the allocation of sewage to three wastewater treatment plants over two weeks. The relaxation is exact, and takes roughly twenty minutes to solve using SOCP~\cite{Boyd1998SOCP}.

\section{Setup}\label{setup}
\subsection{Network modeling}\label{s:def}
The system consists of $s$ well-mixed tanks interconnected by mass flow and diffusion. We denote the set of tanks $\mathcal{S}$.  $V\in\mathbb{R}^{s\times s}$ is a diagonal matrix in which $V_{ii}$ is the volume of tank $i$. Tank $i$ has water inlet flow rate $Q_i^{\textrm{in}}$ and outlet flow rate $Q_i^{\textrm{out}}$. We let $Q_{ij}$ denote the flow from tank $i$ to tank $j$.
Let $d_{ij}$ denote the diffusion between tanks $i$ and $j$, where $d_{ij}=d_{ji}$. Let $C=\textrm{diag}\left[Q_i^{\textrm{in}}\right]$,
\begin{align*}
M_{ij}&=\left\{
\begin{array}{ll}
Q_{ji}, & i\neq j\\
-Q_i^{\textrm{out}}-\sum_{k\in{\mathcal{S}}}Q_{ik},& i= j
\end{array}
\right., \quad 
L_{ij}&=\left\{
\begin{array}{ll}
d_{ij}, & i\neq j\\
-\sum_{k\in{\mathcal{S}}}d_{ik}, & i= j
\end{array}
\right.,
\end{align*}
and $N=M+L$. $M$ and $L$ are respectively compartmental and Laplacian matrices. $M$ is invertible if the network is outflow connected, which is to say that there is a directed path from every tank to some tank with outflow~\cite{jacquez1993qualitative}. Because $L$ is negative semidefinite, $N$ is also invertible if $M$ is outflow connected, and potentially even if $M$ is not outflow connected.

\subsection{Microbial growth}
We model the microbial growth in the tanks using the notation of Section 1.5 of~\cite{bastin2013line}. There are $m$ substrates and biomasses in each perfectly mixed tank. $\xi_i\in\mathbb{R}^m_+$ is the process state vector of tank $i\in\mathcal{S}$, which contains the concentrations of the substrates and biomasses, and $\xi_i^{\textrm{in}}\in\mathbb{R}^m_+$ is the corresponding influent concentration vector. This model is minimal in that $\xi$ includes intermediary products, e.g., substrates produced by one reaction and consumed by another, but not final products such as the $\textrm{CH}_4$ ultimately produced by anaerobic digestion.

There are $r$ different types of reactions that convert substrates to other substrates and biomasses. $\phi_i(\xi_i)\in\mathbb{R}^r_+$ is a vector of the reaction kinetics in tank $i$. We are interested in the case where the elements of $\phi_i(\xi_i)$ are concave functions, and in particular representable as SOC constraints. We show how to do this for Monod and Contois kinetics later in Examples \ref{ex:contois} and \ref{ex:monod}.

Let $\kappa_i\in\mathbb{R}^{m\times r}$ be the stoichiometric matrix relating the reaction vector, $\phi_i(\xi_i)$, to the evolution of the process state in tank $i$. The dynamics in tank $i\in\mathcal{S}$ are
\begin{equation*}
V_{ii}\dot{\xi}_i=V_{ii}\kappa_i\phi_i(\xi_i) -Q_{i}^{\textrm{out}}\xi_i- \sum_{j\in\mathcal{S}}(Q_{ij}+d_{ij})\xi_i
+Q_{i}^{\textrm{in}}\xi_i^{\textrm{in}}+\sum_{j\in\mathcal{S}}(Q_{ji}+d_{ij})\xi_j.
\end{equation*}
The following example illustrates $\xi_i$ and $\kappa_i$.

\begin{example}[Two-step anaerobic digestion]
There are two substrates, $S_i^a$ and $S_i^b$, and two biomasses, $X_i^a$ and $X_i^b$. We let $\xi_i=\left[S_i^a,S_i^b,X_i^a,X_i^b\right]^{\top}$. $S_i^a$ is converted to both $X_i^a$ and $S_i^b$ at the rate $\mu^a\left(S_i^a\right)X_i^a$. $S_i^b$ is converted to $X_i^b$ at the rate $\mu^b\left(S_i^b\right)X_i^b$. Therefore,
\[
\phi_i(\xi_i)=\left[
\begin{array}{c}
\mu^a\left(S_i^a\right)X_i^a\\
\mu^b\left(S_i^b\right)X_i^b
\end{array}
\right]\;\textrm{and}\;
\kappa_i=\left[
\begin{array}{cc}
-1&0\\
1&-1\\
1&0\\
0&1
\end{array}
\right].
\]
\end{example}

We now write the dynamics in vector form. We suppress subscripts to represent stacked vectors, i.e., $\xi=[\xi_1,...,\xi_s]^{\top}$ and $\phi(\xi)=[\phi_1(\xi_1),...,\phi_s(\xi_s)]^{\top}$. Let $A\otimes B$ denote the Kronecker product of $A$ and $B$, $I_\alpha\in\mathbb{R}^{\alpha\times \alpha}$ the identity matrix, and $\hat{A}=A\otimes I_m$. Let $K$ be a block diagonal matrix with $\kappa_1,...,\kappa_s$ on its main diagonal. If $\kappa_i=\kappa$ for all $i\in\mathcal{S}$, then $K=I_s\otimes \kappa$. The dynamics of the full system are given in vector form by
\begin{align}
\hat{V}\dot{\xi}=\hat{V}K \phi(\xi)+\hat{N}\xi+\hat{C}\xi^{\textrm{in}}.\label{dyn}
\end{align}
We allow the dynamics to be non-autonomous, in which case $\hat{N}$, $\hat{C}$, and $\xi^{in}$ can be time-varying.

\subsection{Discretization in time}
To make (\ref{dyn}) compatible with finite-dimensional optimization, we replace the derivatives with a numerical approximation, which we denote $\mathcal{D}_n$. For example, in the case of the implicit Euler method with time step $\Delta$, $\mathcal{D}_n[\xi(\cdot)]=(\xi(n)-\xi(n-1))/\Delta$. $\mathcal{D}_n$ could also be a more sophisticated approximation such as a Runge-Kutta scheme~\cite{betts1998survey}. The time periods are indexed $n\in\mathcal{N}=\{1,...,\tau\}$. We have
\begin{align}
\hat{V}\mathcal{D}_n[\xi(\cdot)]=\hat{V}K  \phi(\xi(n))+\hat{N}(n)\xi(n)+\hat{C}(n)\xi^{\textrm{in}}(n)
\end{align}
for $n\in\mathcal{N}$. The initial condition is $\xi(0)=\xi_0$.

\subsection{Objectives}
We consider objectives of the form
\begin{equation}
\mathcal{F}(\xi,T)=\sum_{n\in\mathcal{N}} \mathcal{F}_{\xi}(\xi(n))+\mathcal{F}_{\phi}(T(n)),
\end{equation}
where $\mathcal{F}_{\xi}$ and $\mathcal{F}_{\phi}$ are convex and $T(n)=\phi(\xi(n))$. The following are examples.
\begin{itemize}
\item Minimizing the outflow of substrates, 
\[
\mathcal{F}_{\xi}(\xi(n))=\sum_{i\in\mathcal{S}}Q_{i}^{\textrm{out}}\eta_i^{\top}\xi_i(n),
\]
where $\eta_i$ is a vector that selects the entries of $\xi_i(n)$ corresponding to pollutants.

\item The production of biogas in a tank is proportional to the kinetics that convert substrates to biomass. Let $\sigma_i\in\mathbb{R}^m_+$ be a vector that is only nonzero for entries of $T_i(n)$ corresponding to biogas production. Let $\mathcal{M}\subseteq \mathcal{S}$ be the subset of tanks that can capture biogas from anaerobic digestion. We maximize biogas through the objective
\[
\mathcal{F}_{\phi}(T(n))=-\sum_{i\in\mathcal{M}}V_{ii}\sigma_i^{\top}T_i(n).
\]

\item Setpoint tracking,
\[
 \mathcal{F}_{\xi}(\xi(n))=\left(\xi(n)-\bar{\xi}\right)^{\top}A\left(\xi(n)-\bar{\xi}\right),
\]
where $A\succeq0$ and $\bar{\xi}$ is a desired operating point.
\end{itemize}

\subsection{Problem statement}\label{statement}
We aim to solve the following optimization problem.
\begin{subequations}
\begin{align}
\mathcal{P}\quad&\min\;\mathcal{F}(\xi,T) \label{P0}\\
\textrm{such that}\quad&T(n)=\phi(\xi(n)),\quad n\in\mathcal{N}\label{P3}\\
&\hat{V}\mathcal{D}_n[\xi(\cdot)]=\hat{V}K  T(n)+\hat{N}(n)\xi(n)
+\hat{C}(n)\xi^{\textrm{in}}(n),\quad n\in\mathcal{N}\label{P1}\\
&\left(\xi,\xi^{\textrm{in}},T\right)\in\Omega.\label{QUps}
\end{align}
\end{subequations}
$\mathcal{P}$ models the optimization of a broad range of bioprocesses such as wastewater treatment. We refer the reader to~\cite{bastin2013line} for broad coverage of this topic.

A solution of $\mathcal{P}$ is a trajectory $\left(\xi(n),\xi^{\textrm{in}}(n),T(n)\right)$, $n\in\mathcal{N}$. The flows and diffusions between tanks, encoded by the matrices $\hat{N}(n)$ and $\hat{C}(n)$, are not decision variables. For this reason constraint (\ref{P1}) is linear. The set $\Omega$ in (\ref{QUps}) consists of linear constraints such as the initial condition, total input matter, and maximum substrate concentrations; several other examples are given in~\cite{taylor2021grad}. Note that $\Omega$ can constrain $T(\cdot)$ so as to allow constraints on the growth without adding nonlinearities. The only nonconvexity is therefore (\ref{P3}), the growth constraint; this is the focus of the next two sections.

We note that in many applications, the flow rates are important decision variables. They are parameters here because our focus is on incorporating the growth kinetics in a convex fashion. In the case that the flow rates are variable, (\ref{P1}) becomes bilinear and hence nonconvex. There are several ways to handle the bilinearity, including
\begin{itemize}
\item McCormick~\cite{mccormick1976computability} and lift-and-project~\cite{Sherali1992bilinear} relaxations;
\item disjunctive programming reformulations if the flow variables are binary, as in~\cite{taylor2021grad};
\item and finding a local minimum via nonlinear programming, e.g., exploiting the biconvex structure with the Alternating Direction Method of Multipliers~\cite{boyd2011distributed}.
\end{itemize}
All three of the above techniques are viable because, as described in the next section, we have a tractable way to represent the growth constraint, (\ref{P3}).

\section{Convex relaxation}\label{relax}
$\mathcal{P}$ is nonconvex because constraint (\ref{P3}) is a nonlinear equality. One way around this difficulty is to instead solve a convex relaxation of $\mathcal{P}$, as in~\cite{taylor2021grad}. If all elements of the vector $\phi(\cdot)$ are concave functions, we obtain a convex relaxation by replacing (\ref{P3}) with the inequality
\begin{equation}
T(n)\leq \phi(\xi(n)),\quad n\in\mathcal{N}.\label{P3R}
\end{equation}
We refer to the resulting optimization as $\mathcal{P}_{\textrm{R}}$. As mentioned earlier, we are interested in the case where (\ref{P3R}) is concave and, ideally, representable as an SOC constraint. Several such examples are given below.

\begin{example}[Contois growth]\label{ex:contois}
Suppose that there is a substrate of concentration $S$, a biomass of concentration $X$, and the growth rate is Contois~\cite{contois1959kinetics}. Then constraint (\ref{P3R}) takes the form
\[
T(n)\leq \frac{\mu S(n)X(n)}{k_{\textrm{C}}X(n)+S(n)},
\]
where $\mu$ and $k_{\textrm{C}}$ are constant parameters. As shown in~\cite{taylor2021grad}, this is concave and can be written as the SOC constraint
\begin{align}
\left\|\left[\begin{array}{c}
\mu S(n)\\
k_{\textrm{C}}T(n)\\
\mu kX(n) 
\end{array}\right]\right\|\leq \mu k_{\textrm{C}}X(n) +  \mu S(n) - kT(n).\label{SOCContois}
\end{align}
\end{example}

\begin{example}[Monod growth with constant biomass]\label{ex:monod}
Consider Example~\ref{ex:contois}, but now suppose that the growth rate is Monod~\cite{monod1949growth}. Then constraint (\ref{P3R}) takes the form
\[
T(n)\leq \frac{\mu S(n)X(n)}{k_{\textrm{M}}+S(n)}.
\]
This constraint is quasiconcave. It becomes concave if we assume that the biomass in each time period is not an optimization variable, but an exogenous parameter, i.e., $X(n)=\bar{X}(n)$ for $n\in\mathcal{N}$. This approximation is often valid because the biomass concentration is typically larger and varies more slowly than the substrate concentrations, and is therefore relatively insensitive to the substrates. In this case, as shown in~\cite{taylor2021grad}, it can be written as the SOC constraint
\begin{align}
\left\|\left[\begin{array}{c}
\mu S(n) \bar{X}(n)\\
k_{\textrm{M}}T(n)\\
\mu k\bar{X}(n) 
\end{array}\right]\right\|\leq \mu k_{\textrm{M}}\bar{X}(n) +  \mu S(n)\bar{X}(n) - kT(n).\label{SOCMonod}
\end{align}
\end{example}

\begin{example}[Interactive and non-interactive growth]\label{ex:interactive}
Suppose the growth of the biomass depends on two rates, $\mu^a(S(n),X(n))$ and $\mu^b(S(n),X(n))$, and the individual kinetics constraints, $T^a(n)\leq\mu^a(S(n),X(n))X(n)$ and $T^a(n)\leq\mu^b(S(n),X(n))X(n)$, both have SOC representations. The dependency often takes one of two forms: non-interactive, $\min\{\mu^a(S(n),X(n)),\mu^b(S(n),X(n))\}$, which in ecological modeling is known as Liebig's Law, and interactive, $\mu^a(S(n),X(n))\mu^b(S(n),X(n))$, which is common in models of bioprocesses.

The non-interactive case is enforced by the two individual kinetics constraints along with $T(n)\leq T^a(n)$ and $T(n)\leq T^b(n)$. The interactive case is in general nonconvex, and does not have an SOC representation. However, the geometric mean of the growth rates, $\sqrt{\mu^a(S(n),X(n))\mu^b(S(n),X(n))}$, does lead to an SOC representation. It is enforced by the two individual kinetics constraints along with $T(n)^2\leq T^a(n)T^b(n)$. The latter is hyperbolic, a type of SOC constraint~\cite{Boyd1998SOCP}.
\end{example}

\section{Exactness}\label{exact}
When (\ref{P3R}) is satisfied with equality, $\mathcal{P}_{\textrm{R}}$ is exact, i.e., has an optimal solution that also solves $\mathcal{P}$; this is the ideal outcome. When (\ref{P3R}) is not satisfied with equality, $\mathcal{P}_{\textrm{R}}$ might still provide a close approximation of $\mathcal{P}$, but this is hard to guarantee. It is therefore useful to have conditions, even if narrow, under which the exactness of $\mathcal{P}_{\textrm{R}}$ is guaranteed.

For the rest of this section we let
\[
\mathcal{D}_n[\xi(\cdot)]=(\xi(n)-\xi(n-1))/\Delta,
\]
which corresponds to the implicit Euler step. We assume that the only constraint specified by $\Omega$ in (\ref{QUps}) is an initial condition, $\xi(0)=\xi_0$, and that $\xi^{\textrm{in}}(\cdot)$ is fixed; note that as long as exactness holds for all feasible values of $\xi^{\textrm{in}}(\cdot)$, it holds when $\xi^{\textrm{in}}(\cdot)$ is a variable. We also assume that strong duality holds for $\mathcal{P}_{\textrm{R}}$; this is a mild assumption that, e.g., holds as long as there is a feasible solution in which $\xi(n)>0$ for all $n\in\mathcal{N}$.

Let $\mathcal{J}(\xi(n))\in\mathbb{R}^{rs\times ms}$ denote the Jacobian matrix of $\phi(\cdot)$ at $\xi(n)$. For convenience, we define the following quantities:
\begin{align*}
\Gamma(n)&=\frac{1}{\Delta}\left(\hat{V}/\Delta-\hat{N}(n)^{\top}-\mathcal{J}(\xi(n))^{\top}K^{\top}\hat{V}\right)^{-1}\hat{V}\\
&=\left(I_{ms}-\Delta\hat{V}^{-1}\left(\hat{N}(n)^{\top}+\mathcal{J}(\xi(n))^{\top}K^{\top}\hat{V}\right)\right)^{-1}\\
\Omega(n)&=-\nabla  \mathcal{F}_{\phi}(T(n))-\Delta K^{\top}\hat{V}\sum_{k=n}^{\tau}\left(\prod_{l=n}^{k}\Gamma(l)\right)\hat{V}^{-1}
\left(\nabla  \mathcal{F}_{\xi}(\xi(k))+\mathcal{J}(\xi(k))^{\top}\nabla  \mathcal{F}_{\phi}(T(k))\right).
\end{align*}
Observe that if $\Delta$ is small enough, $\Gamma(n)$ is positive definite and close to the identify matrix.

\begin{theorem}\label{Thgen}
$\mathcal{P}_{\textrm{R}}$ is exact if at an optimal solution, $\Omega(n)>0$ for all $n\in\mathcal{N}$.
\end{theorem}

\begin{proof}
Let $\lambda(n)\in\mathbb{R}^{ms}$ and $\rho(n)\in\mathbb{R}^{rs}$ be the respective dual multipliers of constraints (\ref{P1}) and (\ref{P3R}) for $n\in\mathcal{N}$. The Lagrangian of $\mathcal{P}_{\textrm{R}}$ is
\begin{align*}
\mathcal{L}&=\mathcal{F}(\xi,T) +\sum_{n\in\mathcal{N}}\rho(n)^{\top}\left(T(n)-\phi(\xi(n))\right)\\
&\quad+\lambda(n)^{\top}\left(\hat{V}(\xi(n-1)-\xi(n))/\Delta + \hat{V}KT(n)
+\hat{N}(n)\xi(n) + \hat{C}(n)\xi^{\textrm{in}}(n) \right).
\end{align*}
Differentiating the Lagrangian by $T(n)$ and $\xi(n)$ and setting it to zero gives
\begin{subequations}
\begin{align}
-\rho(n) &=\nabla  \mathcal{F}_{\phi}(T(n))+ K^{\top}\hat{V}\lambda(n),\quad n\in\mathcal{N},\label{dL1}\\
\mathcal{J}(\xi(n))^{\top}\rho(n)&=\nabla  \mathcal{F}_{\xi}(\xi(n)) -\left(\hat{V}/\Delta-\hat{N}(n)^{\top}\right)\lambda(n)
+\hat{V}\lambda(n+1)/\Delta,\quad n\in\mathcal{N}\setminus\tau \label{dL2}\\
\mathcal{J}(\xi(\tau))^{\top}\rho(\tau)&=\nabla  \mathcal{F}_{\xi}(\xi(\tau)) -\left(\hat{V}/\Delta-\hat{N}(n)^{\top}\right)\lambda(\tau).\label{dL3}
\end{align}
\end{subequations}
We now solve for $\rho(n)$. Premultiplying (\ref{dL1}) by $\mathcal{J}(\xi(n))^{\top}$ and summing with (\ref{dL2}) and (\ref{dL3}) gives
\begin{align*}
\lambda(\tau)=\Delta\Gamma(\tau)\hat{V}^{-1}\left(\nabla  \mathcal{F}_{\xi}(\xi(\tau))+\mathcal{J}(\xi(\tau))^{\top}\nabla  \mathcal{F}_{\phi}(T(\tau))\right),
\end{align*}
and, for $n\in\mathcal{N}\setminus\tau$,
\begin{align*}
\lambda(n)&=\Delta\Gamma(n)\hat{V}^{-1}\left(\nabla  \mathcal{F}_{\xi}(\xi(n))+\mathcal{J}(\xi(n))^{\top}\nabla  \mathcal{F}_{\phi}(T(n))\right)
+\Gamma(n)\lambda(n+1).
\end{align*}
Expanding the recursion yields
\begin{align*}
\lambda(n)&=\Delta\sum_{k=n}^{\tau}\left(\prod_{l=n}^{k}\Gamma(l)\right)\hat{V}^{-1}
\left(\nabla  \mathcal{F}_{\xi}(\xi(k))+\mathcal{J}(\xi(k))^{\top}\nabla  \mathcal{F}_{\phi}(T(k))\right).
\end{align*}
We now substitute this into (\ref{dL1}) to obtain
\begin{align*}
\rho(n)&=-\nabla  \mathcal{F}_{\phi}(T(n))-\Delta K^{\top}\hat{V}\sum_{k=n}^{\tau}\left(\prod_{l=n}^{k}\Gamma(l)\right)\hat{V}^{-1}
\left(\nabla  \mathcal{F}_{\xi}(\xi(k))+\mathcal{J}(\xi(k))^{\top}\nabla  \mathcal{F}_{\phi}(T(k))\right)\\
&=\Omega(n).
\end{align*}
From here we can see that the conditions of the theorem guarantee that $\rho(n)>0$ for all $n\in\mathcal{N}$. 
\end{proof}

Theorem~\ref{Thgen} is of limited immediate use because we must know the optimal solution of $\mathcal{P}_{\textrm{R}}$ to test if $\Omega(n)>0$. It can however be used to derive sufficient conditions for exactness that are easy to test. We now derive two such conditions that do not require knowledge of the optimal solution.

For the rest of this section, assume that $\mathcal{F}_{\xi}$ and $\mathcal{F}_{\phi}$ are linear with gradients $f_{\xi}\in\mathbb{R}^{ms}$ and $f_{\phi}\in\mathbb{R}^{rs}$. In this case, we can write
\begin{align*}
&\Omega(n)=\\
&\quad -\Delta K^{\top}\hat{V}\left(\sum_{k=n}^{\tau}\prod_{l=n}^{k}\Gamma(l)\right)\hat{V}^{-1}f_{\xi}-
\left(I_{ms} +\Delta K^{\top}\hat{V}\sum_{k=n}^{\tau}\left(\prod_{l=n}^{k}\Gamma(l)\right)\hat{V}^{-1}\mathcal{J}(\xi(k))^{\top} \right)f_{\phi}.
\end{align*}

We also assume that each element of the vector of reaction kinetics has bounded slope, so that all entries of $\mathcal{J}(\xi(n))$, $n\in\mathcal{N}$, are bounded.

Let
\[
\Psi(n)=\hat{V}^{-1}\hat{N}(n)^{\top}+\hat{V}^{-1}\mathcal{J}(\xi(n))^{\top}K^{\top}\hat{V}.
\]
The first term of $\Psi(n)$ is negative semidefinite because $N(n)$ is compartmental~\cite{jacquez1993qualitative}. The latter term is bounded by assumption, and as we will see in the examples, usually negative semidefinite---we assume that this is the case. We therefore assume that the eigenvalues of $\Psi(n)$ are in the range $[-\bar{\psi},0]$, where $\bar{\psi}>0$ is an upper bound on the magnitude. We can use the push-through identity to write
\begin{align*}
\Gamma(n)&=I_{ms}+\Delta\Psi(n)(I_{ms}-\Delta\Psi(n))^{-1}.
\end{align*}
The eigenvalues of the second term are in the range $[-\Delta\bar{\psi},0]$, and the eigenvalues of $\Gamma(n)$ are in the range of $[1-\Delta\bar{\psi},1]$.

\begin{corollary}\label{Corxi}
If $f_{\phi}=0$ and $K^{\top}f_{\xi}<0$, then there exists a $\Delta>0$ for which $\mathcal{P}_{\textrm{R}}$ is exact.
\end{corollary}
\begin{proof}
(Sketch)
Observe that
\[
\prod_{l=n}^{k}\Gamma(l)
\]
is equal to $I_{ms}$ plus terms that are norm-bounded by positive powers of $\Delta\bar{\psi}$. Similarly,
\[
\sum_{k=n}^{\tau}\prod_{l=n}^{k}\Gamma(l)
\]
is equal to $(\tau-n+1)I_{ms}$ plus terms that are norm-bounded by positive powers of $\Delta\bar{\psi}$. We can make these terms arbitrarily small by choosing $\Delta$ small. We therefore write
\[
\sum_{k=n}^{\tau}\prod_{l=n}^{k}\Gamma(l)\approx(\tau-n+1)I_{ms}.
\]
Because $f_{\phi}=0$, we have that
\begin{align*}
\Omega(n)\approx&-(\tau-n+1)\Delta K^{\top}f_{\xi}.
\end{align*}
This is strictly positive for all $n\in\mathcal{N}$ if $K^{\top}f_{\xi}<0$.
\end{proof}

\begin{corollary}\label{Corphi}
If $f_{\xi}=0$ and $f_{\phi}>0$, then there exists a $\Delta>0$ for which $\mathcal{P}_{\textrm{R}}$ is exact.
\end{corollary}
\begin{proof}
(Sketch)
Following the same logic as Corollary~\ref{Corxi}, we can choose $\Delta>0$ such that 
\begin{align}
\Omega(n)\approx&\left(I_{ms}-\Delta K^{\top}\sum_{k=n}^{\tau}\mathcal{J}(\xi(k))^{\top}\right)f_{\phi}.\label{OmegaCorphi}
\end{align}
The second term in the parentheses can be made arbitrarily small by choosing $\Delta$ to be small. In this case
$\Omega(n)\approx f_{\phi}$, which is positive if $f_{\phi}>0$.
\end{proof}

A shortcoming of Corollaries~\ref{Corxi} and~\ref{Corphi} is that they can be limited to short time intervals. This is because if an interval is to remain constant, the number of time periods, $\tau$, must increase as the step, $\Delta$, decreases, and the approximations in the proofs of the corollaries do not hold for increasing $\tau$. On the other hand, these are conservative sufficient conditions. For example, the second term in the parentheses of (\ref{OmegaCorphi}) will often be positive semidefinite or nearly so. This is because it is typical for $K$ to be lower triangular with a negative diagonal and for $\mathcal{J}(\xi(n))$ to be nonnegative and nearly diagonal. We therefore expect that exactness will sometimes hold for larger $\Delta$ over longer time intervals.

We view these theoretical results not as a complete characterization of when exactness is guaranteed, but rather as evidence that $\mathcal{P}_\textrm{R}$ is exact for a meaningful set of problems, and as guidance as to how to identify them. While there are certainly problems of interest for which their conditions do not hold, $\mathcal{P}_\textrm{R}$ may nonetheless provide a useful and sometimes perfect approximation.

\subsection{Steady state}
It may be of interest to optimize (\ref{dyn}) in steady state, e.g., when the solution does not change significantly on the timescale of interest, to find the best operating point, or to reduce the number of variables. We obtain a steady state optimization by dropping the time index and replacing the finite difference in (\ref{P1}) with zero. The resulting (relaxed) optimization is:
\begin{subequations}
\begin{align}
\mathcal{P}_{\textrm{RS}}\quad&\min\;\mathcal{F}(\xi,T) \label{P0s}\\
\textrm{such that}\quad&T\leq\phi(\xi)\label{P3s}\\
&0=\hat{V}K  T+\hat{N}\xi+\hat{C}\xi^{\textrm{in}}\label{P1s}\\
&\left(\xi,\xi^{\textrm{in}},T\right)\in\Omega.\label{QUpss}
\end{align}
\end{subequations}

We remark that, in general, a solution to $\mathcal{P}_{\textrm{RS}}$ is not guaranteed to be an equilibrium of (\ref{dyn}). One special case in which guarantees do exist is the gradostat, which we discuss in Example~\ref{app:gradostat}. We refer the reader to~\cite{taylor2021grad} for a brief summary.

\begin{theorem}\label{Thgenss}
$\mathcal{P}_{\textrm{RS}}$ is exact if the network is outflow connected and at the optimal solution,
\begin{align*}
0<&\left(I_{ms}+K^{\top}\hat{V}\left(\hat{N}^{\top}\right)^{-1}\mathcal{J}(\xi)^{\top}\right)^{-1}
\left(K^{\top}\hat{V}\left(\hat{N}^{\top}\right)^{-1}\nabla  \mathcal{F}_{\xi}(\xi)-\nabla  \mathcal{F}_{\phi}(T)\right).
\end{align*}
\end{theorem}

\begin{proof}
The Lagrangian of $\mathcal{P}_{\textrm{RS}}$ is
\begin{align*}
\mathcal{L}&=\mathcal{F}(\xi,T) +\rho^{\top}\left(T-\phi(\xi)\right)+\lambda^{\top}\left( \hat{V}KT+\hat{N}\xi + \hat{C}\xi^{\textrm{in}} \right).
\end{align*}
Differentiating the Lagrangian and setting it to zero gives
\begin{align*}
0&=\nabla  \mathcal{F}_{\phi}(T) +\rho + K^{\top}\hat{V}\lambda\\
\mathcal{J}(\xi)^{\top}\rho &=\nabla  \mathcal{F}_{\xi}(\xi) + \hat{N}^{\top}\lambda.
\end{align*}

We now solve for $\rho$. $\hat{N}$ is invertible due to outflow-connectedness. Then
\begin{align*}
\lambda&=\left(\hat{N}^{\top}\right)^{-1}\left(\mathcal{J}(\xi)^{\top}\rho - \nabla  \mathcal{F}_{\xi}(\xi) \right),
\end{align*}
and
\begin{align*}
0&=\nabla  \mathcal{F}_{\phi}(T) +\rho + K^{\top}\hat{V}\left(\hat{N}^{\top}\right)^{-1}\left(\mathcal{J}(\xi)^{\top}\rho - \nabla  \mathcal{F}_{\xi}(\xi) \right).
\end{align*}
Solving, we have
\begin{align*}
\rho&=\left(I_{ms}+K^{\top}\hat{V}\left(\hat{N}^{\top}\right)^{-1}\mathcal{J}(\xi)^{\top}\right)^{-1}
\left(K^{\top}\hat{V}\left(\hat{N}^{\top}\right)^{-1}\nabla  \mathcal{F}_{\xi}(\xi)-\nabla  \mathcal{F}_{\phi}(T)\right).
\end{align*}
By complementary slackness, $\mathcal{P}_{\textrm{RS}}$ is exact when $\rho>0$. 
\end{proof}

As in the latter part of the previous section, we now assume that $\mathcal{F}_{\xi}$ and $\mathcal{F}_{\phi}$ are linear with gradients $f_{\xi}\in\mathbb{R}^{ms}$ and $f_{\phi}\in\mathbb{R}^{rs}$. 

\begin{corollary}\label{Corss}
Suppose that for all $\xi\geq0$,
\[
\left(I_{ms}+K^{\top}\hat{V}\left(\hat{N}^{\top}\right)^{-1}\mathcal{J}(\xi)^{\top}\right)^{-1}\geq 0,
\]
and 
\[
K^{\top}\hat{V}\left(\hat{N}^{\top}\right)^{-1}f_{\xi}-f_{\phi}\geq0.
\]
If at least one of the inequalities is strict, then $\mathcal{P}_{\textrm{RS}}$ is exact.
\end{corollary}

Like Theorem~\ref{Thgenss}, Corollary~\ref{Corss} depends on the optimal solution, but to a lesser extent. Whereas Theorem~\ref{Thgenss} depends on $\xi$ and $T$ through the objective and $\mathcal{J}(\xi)$, Corollary~\ref{Corss} only depends on $\xi$ through $\mathcal{J}(\xi)$. This is more manageable because $\mathcal{J}(\xi)$ is often positive and nearly diagonal. For example, if we assume assume that biomass is constant and all growth rates are Monod, as in Example~\ref{ex:monod}, then $\mathcal{J}(\xi)$ is positive on the diagonal and zero elsewhere.


\section{Examples}\label{app}

\subsection{The gradostat}\label{app:gradostat}
The gradostat is a special case of (\ref{dyn}) where in each tank $i\in\mathcal{S}$, a single substrate, $S_i$, is converted to a single type of biomass, $X_i$. The conversion occurs at the rate $\phi_i(S_i,X_i)/y$, where $y$ is the yield. Then $\xi_i=[S_i,X_i]^{\top}$ and $\kappa_i=[-1/y,1]^{\top}$. In~\cite{taylor2021grad}, several examples are given in which the gradostat is optimized over time and in steady state.

Here we first apply Corollary~\ref{Corxi} to the gradostat. Suppose $\mathcal{F}_{\xi}(\xi(n))=f_S^{\top}S(n)+f_X^{\top}X(n)$. In this case, Corollary~\ref{Corxi} holds if $-f_S/y+f_X<0$. When dealing with the decontamination of undesirable substrates, we do not seek to minimize biomass, in which case $f_X=0$ and the condition is satisfied if $f_S>0$.

We now apply Theorem~\ref{Thgenss} to the gradostat in steady state. Let $\mathcal{J}_S(\xi)\in\mathbb{R}^{r\times s}$ be the Jacobian matrix of $\phi(\xi)$ with respect to $S$, and let $\mathcal{J}_X(\xi)\in\mathbb{R}^{r\times s}$ be the Jacobian matrix of $\phi(\xi)$ with respect to $X$. $\mathcal{P}_{\textrm{RS}}$ is exact if
\begin{align*}
&\left(I_s+V\left(N^{\top}\right)^{-1}\left(-\mathcal{J}_S(\xi)^{\top}/y+\mathcal{J}_X(\xi)^{\top}\right)\right)^{-1}
\left(K^{\top}V\left(N^{\top}\right)^{-1}\nabla  \mathcal{F}_{\xi}(\xi)-\nabla  \mathcal{F}_{\phi}(T)\right)>0 .
\end{align*}
It is straightforward to verify that when $\mathcal{F}_{\xi}(\xi)=0$, this condition directly implies Theorem 1 in~\cite{taylor2021grad}. Similarly, we obtain Corollary 1 in~\cite{taylor2021grad} if we specialize Corollary~\ref{Corss} to the gradostat.

\subsection{Wastewater treatment}\label{app:ww}
In this example, we optimize an idealized wastewater treatment system consisting of three wastewater treatment plants of the city of Paris and its suburbs. The model is based on that in~\cite{ROBLESRODRIGUEZ2018880,robles2019management}, and the influent data from the \texttt{Inf\_rain\_2006} dataset of~\cite{IWABSM1}. We implemented the model using the parser CVX~\cite{cvx} and the solver Gurobi~\cite{gurobi} on a personal computer from 2014 with a 1.4 GHz dual-core processor.

In the present study, the flow rates are considered constant and given by $Q^{\textrm{in}}_1=0.1\;m^3/s$, $Q^{\textrm{in}}_2=0.4\;m^3/s$, and $Q^{\textrm{in}}_3=0.2\;m^3/s$. All tanks have volume $1000\;m^3$. In each plant $i\in\mathcal{S}$, the state vector $\xi_i=\left[\xi^{\textrm{BOD}}_i,\xi^{\textrm{NH}_4^+}_i,\xi^{\textrm{NO}_2^-}_i,\xi^{\textrm{NO}_3^-}_i\right]^{\top}\in\mathbb{R}^4$ consists of biochemical oxygen demand, ammonia nitrogen, nitrite, and nitrate. The biomass in each time period is assumed to be an exogenous parameter, and therefore not a component of the process state. This is an admissible assumption in the sense that the biomass concentration is typically much slower than that of the other process components, and has a larger amplitude.

In each tank $i\in\mathcal{S}$, the elements of the process state have kinetics $\phi_i^{\textrm{BOD}},\phi_i^{\textrm{NH}_4^+},\phi_i^{\textrm{NO}_2^-}$, and $\phi_i^{\textrm{NO}_3^-}$. All assume Monod growth rates with parameters given in Table~\ref{GrowthPar}, which comes from Table~1 in~\cite{Ayed2018MS}. Because the biomass is constant, the growth kinetics can be represented as SOC constraints in $\mathcal{P}_{\textrm{R}}$.
\begin{table}[h]
	\centering
	\begin{tabular}{|l|c|c|c|}
	\hline
	Parameter & Plant 1 & Plant 2 & Plant 3  \\
	\hline
	$\mu^{\textrm{BOD}}$ (1/day) & $3.99$ & $2.56$ & $1.93$\\
	$\mu^{\textrm{NH}_4^+}$ & $0.84$ & $0.83$ & $0.89$\\
	$\mu^{\textrm{NO}_2^-}$ & $1.68$ & $1.27$ & $0.92$\\
	$\mu^{\textrm{NO}_3^-}$ & $1.21$ & $1.38$ & $0.85$\\
	\hline
	$K^{\textrm{BOD}}$ (mg/L) & $13.67$ & $11.65$ & $14.26$\\
	$K^{\textrm{NH}_4^+}$ & $6.59$ & $14.98$ & $8.53$\\
	$K^{\textrm{NO}_2^-}$ & $2.46$ & $1.15$ & $2.55$\\
	$K^{\textrm{NO}_3^-}$ & $1.40$ & $2.69$ & $4.20$\\
	\hline
	$y^{\textrm{NH}_4^+,\textrm{NO}_2^-}$ & $0.28$ & $0.25$ & $0.27$\\
	$y^{\textrm{NO}_2^-,\textrm{NO}_3^-}$ & $0.68$ & $0.64$ & $0.70$\\
	\hline
	\end{tabular}
	\caption{Growth function parameters}
	\label{GrowthPar}
\end{table}
The stochiometric matrix for each plant $i\in\mathcal{S}$ is
\[
\kappa_i=\left[
\begin{array}{cccc}
-1&0 & 0 & 0\\
0 & -1 & 0 & 0\\
0 & 1/y_i^{\textrm{NH}_4^+,\textrm{NO}_2^-} & -1 & 0\\
0 & 0 & 1/y_i^{\textrm{NO}_2^-,\textrm{NO}_3^-} & -1
\end{array}
\right].
\]
We used the implicit Euler method, $\mathcal{D}_n[\xi(\cdot)]=(\xi(n)-\xi(n-1))/\Delta$, with the stepsize $\Delta=1$, which corresponds to 15 minutes. There are $\tau=1345$ time periods, so that the total time is two weeks. The boundary condition is $\xi(0)=\xi(\tau)$. This could represent periodic operation, or exogenous conditions that are similar from week to week.

The process state must satisfy $\xi^{\textrm{BOD}}_i(n)\leq 150$ mg/L and $\xi^{\textrm{NH}_4^+}_i(n)\leq 60$ mg/L for each $i\in\mathcal{S}$ and $n\in\mathcal{N}$. Without these constraints, most of the substrate would be directed to Plant 1, which is more efficient than the others. This constraint could represent, for example, regulatory limits on the pollution released by the plants.

In each time period, fixed quantities of $\textrm{BOD}$ and $\textrm{NH}_4^+$, $\Xi^{\textrm{BOD}}(n)$ and $\Xi^{\textrm{NH}_4^+}(n)$, must be allocated over the treatment plants. These quantities are based on the \texttt{Inf\_rain\_2006} dataset of~\cite{IWABSM1}, which covers 1345 15-minute intervals. $\Xi^{\textrm{BOD}}(\cdot)$ is set to $S_{\textrm{S}}$ (readily biodegradable substrate), and $\Xi^{\textrm{NH}_4^+}(\cdot)$ to $S_{\textrm{NH}}$ ($\textrm{NH}_4^+$ and $\textrm{NH}_3$ nitrogen) in~\cite{IWABSM1}. The allocation is represented by the linear constraints
\begin{align*}
\Xi^{\textrm{BOD}}(n)&=\sum_{i=1}^3Q^{\textrm{in}}_i(n)\xi^{\textrm{BOD},\textrm{in}}_i(n)\\
\Xi^{\textrm{NH}_4^+}(n)&=\sum_{i=1}^3Q^{\textrm{in}}_i(n)\xi^{\textrm{NH}_4^+,\textrm{in}}_i(n),
\end{align*}
for each $n\in\mathcal{N}$. The other influent concentrations are $\xi^{\textrm{NO}_2^-,\textrm{in}}_i(n)=3$ and $\xi^{\textrm{NO}_3^-,\textrm{in}}_i(n)=10$ for $i\in\mathcal{S}$ and $n\in\mathcal{N}$. 

The plant biomass concentrations are set to $\bar{X}_i(n)=100\left(1+(-1)^i\sin(10\pi n/\tau)\right)$ for $i\in\mathcal{S}$ and $n\in\mathcal{N}$. Observe that this has larger magnitude and varies slower than the other influents.

For each $i\in\mathcal{S}$ and $n\in\mathcal{N}$, the optimization variables are $\xi_i(n)$, $\xi^{\textrm{BOD},\textrm{in}}_i(n)$, and $\xi^{\textrm{NH}4,\textrm{in}}_i(n)$. The objective is to minimize the untreated wastewater released by the plants, given by
\[
\sum_{n\in\mathcal{N}}\sum_{i\in\mathcal{S}}Q^{\textrm{out}}_i\eta_i^{\top}\xi_i(n),
\]
where we note that $Q^{\textrm{out}}_i=Q^{\textrm{in}}_i$, and $\eta_i=[2,2,0.3,0.1,0]^{\top}$. This objective was chosen to satisfy Corollary~\ref{Corxi}. Note, however, that the result does not fully apply due to the boundary condition, $\xi(0)=\xi(\tau)$, and the upper limit on the process state.

The convex relaxation $\mathcal{P}_{\textrm{R}}$ contains 145272 variables (in standard form) and took 17 minutes to solve. The 18 solver iterations accounted for only four seconds, and the rest of the time was used for preprocessing. Despite not fully satisfying Corollary~\ref{Corxi}, the solution was exact in all time periods.

Figures~\ref{fig:WWSin} and~\ref{fig:WWS} respectively show the optimal influent allocation, $\xi^{\textrm{in}}(\cdot)$, and the resulting plant effluent concentrations, $\xi(\cdot)$, between hours 175 and 275.
 \begin{figure}[h]
 \centering
\includegraphics[width=0.9\columnwidth]{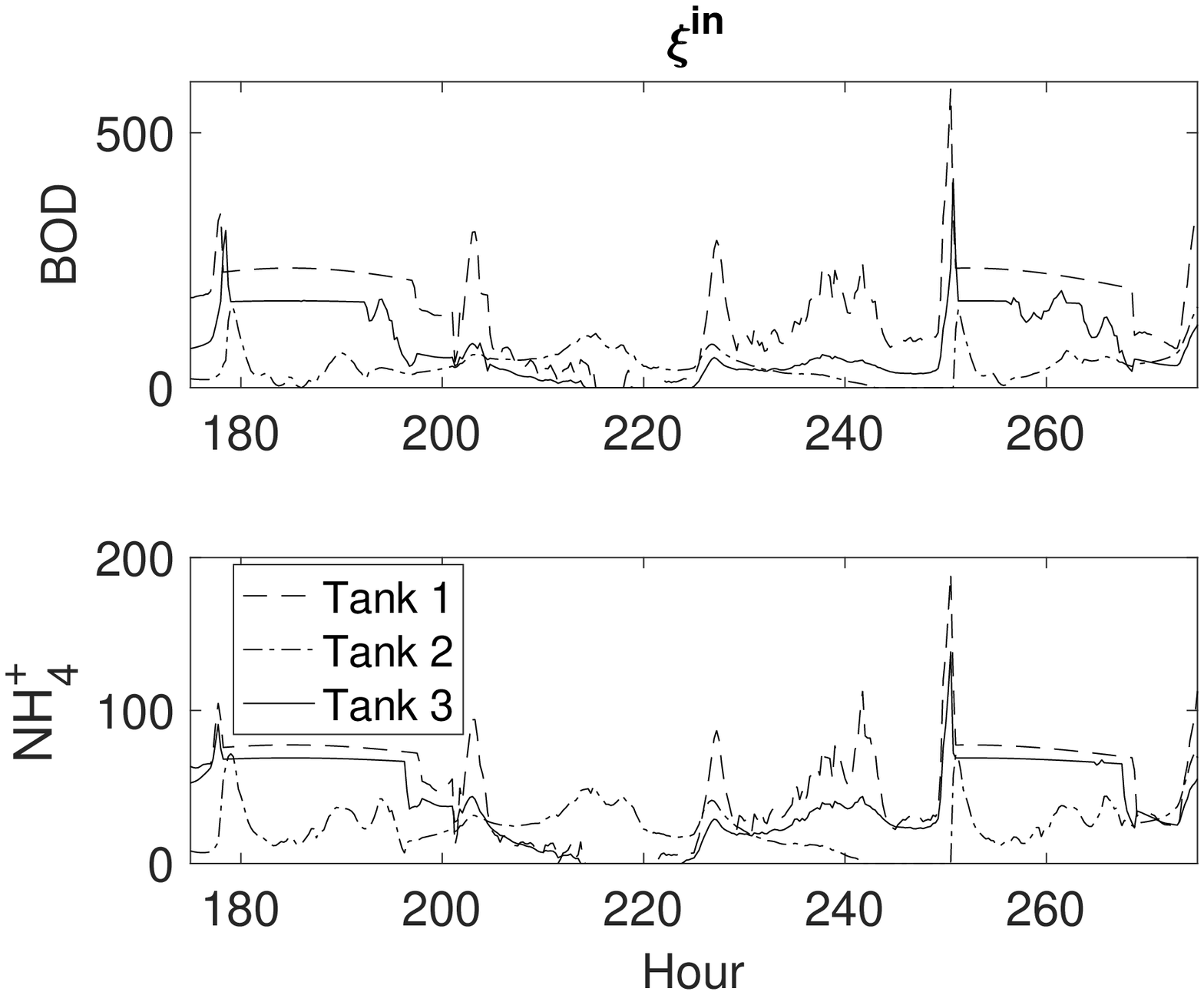} 
\caption{$\xi^{\textrm{in}}(\cdot)$ from hour 175 to 275. The units are mg/L.}
\label{fig:WWSin}
\end{figure}
 \begin{figure}[h]
 \centering
\includegraphics[width=0.9\columnwidth]{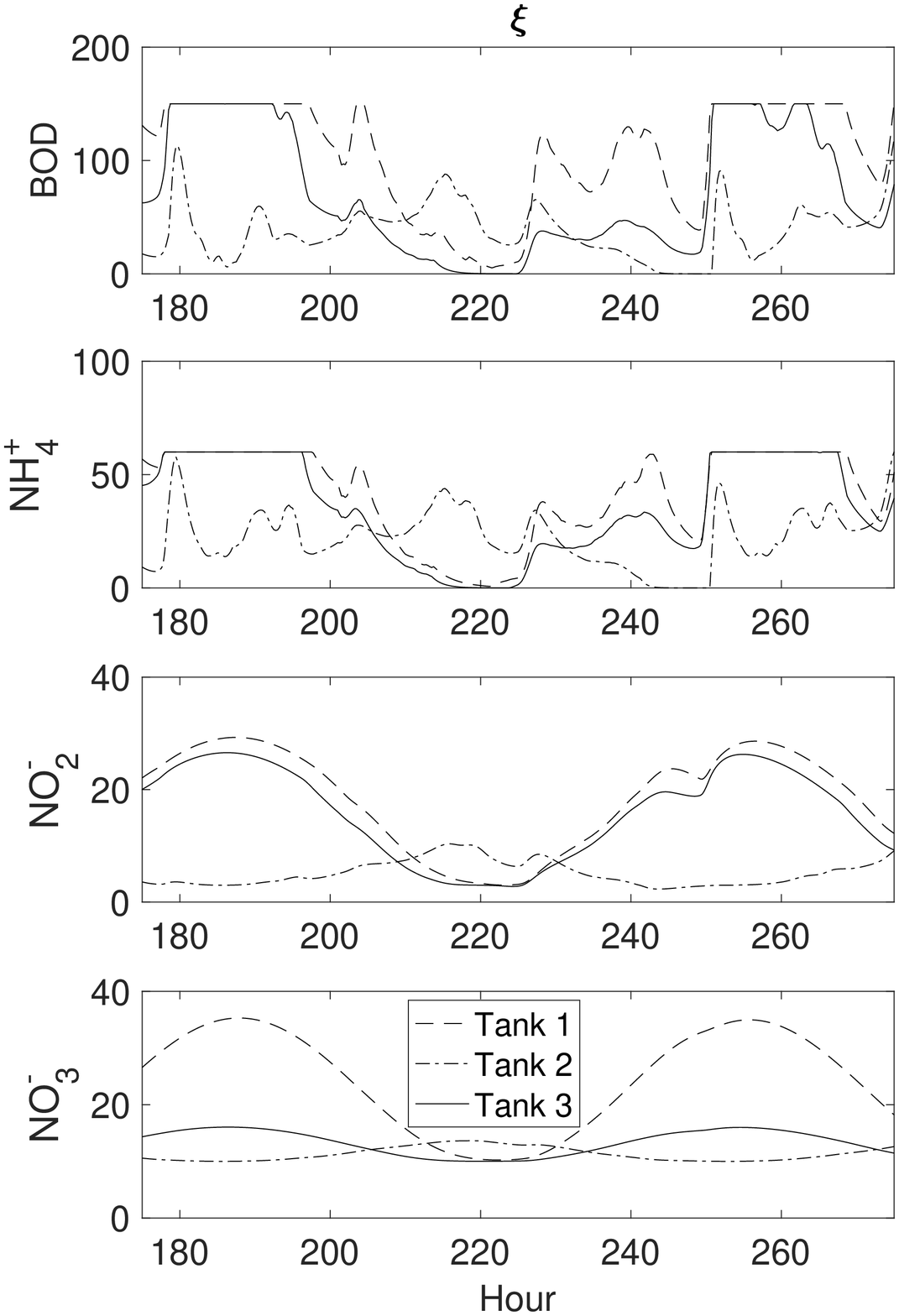} 
\caption{$\xi(\cdot)$ from hour 175 to 275. The units are mg/L.}
\label{fig:WWS}
\end{figure}
The slower variation of the biomass dominates that of the diurnal variation in the \texttt{Inf\_rain\_2006} dataset, leading to concentration increases whenever the biomass influent into each plant is high. There is a spike in $\Xi^{\textrm{BOD}}(\cdot)$ and $\Xi^{\textrm{NH}_4^+}(\cdot)$ around hour 250. This causes $\xi^{\textrm{BOD}}_i(\cdot)$, and $\xi^{\textrm{NH}4}_i(n)$ to hit their concentration limits in Plants 1 and 3. When this happens, the remainder is allocated to Plant 2, which has little biomass at the time, and hence cannot transform the substrates as efficiently.

\section{Conclusion}
We have formulated a convex relaxation for optimizing a broad class of bioprocesses. We proved that the relaxation is exact under simple conditions, and implemented it on a wastewater treatment example with over one hundred thousand variables. We believe that a wide range of problems can be approached in this manner due to the generality of the model and the tractability of SOCP.

One direction of future work is dealing with inexactness. Two options are deriving general convex underestimators to limit the relaxation gap, as in~\cite{taylor2021grad}, and finding local minima of the non-relaxed problem. In particular, the concave-convex procedure~\cite{yuille2003concave} is well-suited to the nonconvexity encountered here and would entail solving a sequence of SOCPs. We also intend to incorporate new elements into the model such as biomass death and recirculation, and to apply the relaxation in other contexts such as enzymes, where Michaelis-Menten kinetics~\cite{michaelis1913kinetik} have the same form as Monod kinetics.


\section*{Acknowledgments}
Funding is acknowledged from the Natural Sciences and Engineering Research Council of Canada and the French LabEx NUMEV (Project ANR-10 LABX-20), incorporated into the I-Site MUSE, which partially funded the sabbatical of J.A.~Taylor at MISTEA lab.

\bibliography{MainBib}

\end{document}